\documentclass[11pt]{amsart}
\usepackage{graphicx}
\usepackage{setspace}
\usepackage{amssymb,amsthm,amsmath,amstext}
\usepackage{mathrsfs}  % allows nice script font for sheaves
\usepackage{bm}        % allows bold italic letters in math mode
\usepackage{mathtools} % allows more extendible arrows
\usepackage{stmaryrd}  % additional nice fonts

\usepackage[left=1.5in,top=0.95in,right=1.5in,bottom=0.9in]{geometry}

\usepackage[all]{xy}

\theoremstyle{plain}
\newtheorem{theorem}{Theorem}[section]
\newtheorem*{theorem*}{Theorem}

\newtheorem{prop}[theorem]{Proposition}
\newtheorem{lemma}[theorem]{Lemma}

\newtheorem{comp}[theorem]{Computation}

\newtheorem{theoremintro}{Theorem}

\newtheorem{problemintro}{Problem}

\theoremstyle{definition}

\newtheorem{algorithm}[theorem]{Algorithm}

\theoremstyle{remark}
\newtheorem{remark}[theorem]{Remark}

\newcommand{\CH}{\mathrm{CH}}

\newcommand{\Sym}{\mathrm{Sym}}
\newcommand{\isom}{\cong}

\newcommand{\Aut}{\mathrm{Aut}}
\DeclareMathOperator{\rk}{rk}
\newcommand{\sing}{\mathrm{sing}}
\newcommand{\sm}{\mathrm{sm}}

\renewcommand{\P}{\mathbb{P}}
\newcommand{\Z}{\mathbb{Z}}
\newcommand{\Q}{\mathbb{Q}}
\newcommand{\C}{\mathbb{C}}
\newcommand{\F}{\mathbb{F}}

\newcommand{\Group}[1]{\mathrm{#1}}
\newcommand{\GL}{\Group{GL}}
\newcommand{\PGL}{\Group{PGL}}

\newcommand{\Gm}{\mathbb{G}_m}

\DeclarePairedDelimiter\abs{\lvert}{\rvert}

\newcommand{\inv}{^{-1}}

\usepackage[backref=page]{hyperref}
\hypersetup{pdftitle={A census of cubic fourfolds over $\F_2$}}
\hypersetup{pdfauthor={Asher Auel, Avinash Kulkarni, Jack Petok, Jonah
Weinbaum}}
\hypersetup{colorlinks=true,linkcolor=blue,anchorcolor=blue,citecolor=blue}

\renewcommand{\:}{\colon}

\newcommand{\wtilde}{\widetilde}

\newcommand{\numberlist}[2][0.9\linewidth]{%
  \parbox[t]{#1}{\printcommalist{#2}}%
}

\newcommand{\printcommalist}[1]{%
  \begingroup\lccode`~=`,\lowercase{\endgroup\def~}{\mathcomma\penalty0 }%
  \mathcode`,="8000
  \thinmuskip=6mu plus 6mu minus 2mu
  $#1$
}
\mathchardef\mathcomma=\mathcode`,

\makeatletter
\@namedef{subjclassname@2020}{\textup{2020} Mathematics Subject Classification}
\makeatother

\begin{document}

\title[A census of cubic fourfolds over $\F_2$]{A census of cubic fourfolds over $\F_2$} 

\author{Asher Auel}
\author{Avinash Kulkarni}
\author{Jack Petok}
\author{Jonah Weinbaum}

\address{Department of Mathematics, 
Dartmouth College, Hanover, New Hampshire}

% \address{\texttt{\it E-mail
% addresses: \tt asher.auel@dartmouth.edu,
% avinash.a.kulkarni@dartmouth.edu, \newline 
% \hspace*{1.13in} jack.petok@dartmouth.edu, jonah.r.weinbaum.23@dartmouth.edu}}

\address{\texttt{\it E-mail
addresses: \tt asher.auel@dartmouth.edu}}
\address{\texttt{\it E-mail
address: \tt avinash.a.kulkarni@dartmouth.edu}}
\address{\texttt{\it E-mail
address: \tt jack.petok@dartmouth.edu}}
\address{\texttt{\it E-mail
address: \tt jonah.r.weinbaum.23@dartmouth.edu}}

%\date{\today}

% \subjclass[2020]{Primary 11M38: Zeta and $L$-functions in characteristic $p$, 14Q10: Surfaces, hypersurfaces, 14J70 Hypersurfaces and algebraic geometry; Secondary 14E08: Rationality questions, 14J28: $K3$ surfaces and Enriques surfaces
% }

\begin{abstract}
  We compute a complete set of isomorphism classes of cubic fourfolds
  over $\F_2$.  Using this, we are able to compile statistics about
  various invariants of cubic fourfolds, including their counts of
  points, lines, and planes; all zeta functions of the smooth cubic
  fourfolds over $\F_2$; and their Newton polygons.  One particular
  outcome is the number of smooth cubic fourfolds over $\F_2$, which
  we fit into the asymptotic framework of discriminant complements.
  Another motivation is the realization problem for zeta functions of
  $K3$ surfaces.  We present a refinement to the standard method
  of orbit enumeration that leverages filtrations and gives a
  significant speedup.  In the case of cubic fourfolds, the relevant
  filtration is determined by Waring representation and the method
  brings the problem into the computationally tractable range.
\end{abstract}

\maketitle

\vspace*{-.5cm}

\section*{Introduction}
\label{sec:introduction}

The study of cubic fourfolds over finite fields (e.g.,
\cite{addington_auel}, \cite{ABBV}, \cite{charles_pirutka},
\cite{debarre_laface_roulleau}, etc.) has grown as a respectable side
industry supporting the main threads of investigation---the
rationality problem and its connections to derived categories,
algebraic cycles, K3 surfaces, and hyperk\"ahler varieties---of
complex cubic fourfolds. In this paper and its accompanying
code~\cite{github},
%(\url{https://github.com/JonahWeinbaum/cubic-fourfolds}), 
we generate a database of all cubic fourfolds over $\F_2$ up to
isomorphism. We also compute many of their most important invariants,
including their automorphism groups, their point counts, and
information about their algebraic cycles.  In particular, we can
report the following.

\begin{theoremintro}
\label{thm:smooth}
Of the $3\, 718\, 649$ isomorphism classes of cubic fourfolds over
$\F_2$, exactly $1\, 069\, 562$ are smooth, of which $533\, 262$ are
ordinary, $8688$ are supersingular, $107\, 552$ are Noether--Lefschetz
general, and $702\, 153$ contain a plane.  The smooth cubic fourfolds
admit $86\, 472$ distinct zeta functions.
\end{theoremintro}

The algorithmic methods to generate our database of cubic fourfolds
are of independent interest: we present a new technique for
enumerating a complete set of orbit representatives of a finite group
$G$ acting linearly on a high-dimensional vector space $V$ over a
finite field that leverages $G$-stable filtrations of $V$.  In the
case of cubic fourfolds over $\F_2$, the relevant action is the
representation of $G=\GL_6(\F_2)$ on the $56$-dimensional
$\F_2$-vector space $V=\text{Sym}^3(\F_2^6)$ of homogeneous cubic
forms in six variables.  In certain situations, our method provides a
substantial speedup over complete orbit partition methods, such as
union-find.  The advantage of our method, assuming the existence of
good $G$-stable filtrations, is that we do not need to iterate through
every element of $V$.
A complexity analysis in \S\ref{ss:complexity} shows that under
favorable situations our method is linear in the number of orbits,
which is asymptotically optimal; in the case of cubic fourfolds over
$\F_2$, our method gives a roughly square-root speedup.

Our work on cubic fourfolds was partially inspired by Kedlaya and
Sutherland's census~\cite{kedlaya_sutherland-census} of quartic $K3$
surfaces over $\F_2$. There, a complete partition of quartic surfaces
into $\GL_4(\F_2)$-orbits was achieved in a few days on a powerful
computer; with our method, it takes $3$ minutes on a laptop to compute
a complete set of orbit representatives.  They also compute the zeta
functions of the smooth orbits, as well as a longer list of
potential zeta functions of $K3$ surfaces over $\F_2$.  This is
achieved by enumerating the candidate Weil polynomials on the middle
$\ell$-adic
cohomology~\cite[Computation~3(c)]{kedlaya_sutherland-census}.
Kedlaya and Sutherland pose the following.

\begin{problemintro}
\label{prop:HT}
Determine the set of zeta functions of $K3$ surfaces defined
over $\F_2$.
\end{problemintro}

We remark that the Tate conjecture for $K3$ surfaces (proved by
\cite{maulik}, \cite{charles1}, \cite{charles2}, \cite{madapusipera},
\cite{kim_madapusipera}, \cite{madapusipera-erratum},
\cite{ito_ito_koshikawa}), implies that there are finitely many
isomorphism classes of $K3$ surfaces defined over a fixed finite field
by the work of Lieblich, Maulik, and
Snowden~\cite{lieblich_maulik_snowden}, which holds in any
characteristic.  A resolution of Problem~\ref{prop:HT} would provide a
kind of Honda--Tate theory for $K3$ surfaces. The work of
Taelman~\cite{taelman} implies that the transcendental part of every
Weil polynomial in~\cite[Computation~3(c)]{kedlaya_sutherland-census}
is expected to arise from some $K3$ surface defined over a
\emph{suitable extension} of the base field, but we are interested in
which zeta functions arise from $K3$ surfaces over $\F_2$.

Over the complex numbers, cubic fourfolds are Fano varieties of {\it
$K3$ type}, with the Hodge structures on their middle cohomology
resembling those of $K3$ surfaces.  Hassett~\cite{hassett} classified
those cubic fourfolds that admit Hodge-theoretically {\it associated
K3 surfaces}, namely the {\it admissible} special cubic fourfolds.
Over a finite field, the Weil polynomial on the middle dimensional
$\ell$-adic cohomology of a special cubic fourfold has a factor (the nonspecial Weil polynomial) that looks like the Weil
polynomial of a $K3$ surface, and we would expect this polynomial to
be realizable by a $K3$ surface defined over $\F_2$ whenever a
Hodge-theoretically associated $K3$ surface is defined over $\F_2$.
Thus our computation of the zeta functions of cubic fourfolds
(see~\S\ref{sec:zeta}) provides many new Weil polynomials that should
arise from $K3$ surfaces, and fertile ground for the arithmetic study
of the associated $K3$ surface.  In cases where there is an explicit
algebraic construction of an associated $K3$ surface, for example, for
cubic fourfolds containing a plane, the nonspecial Weil polynomial of
the cubic is the primitive Weil polynomial of some $K3$ surface over
$\F_2$.  On the other hand, our census exhibits many explicit special
cubic fourfolds that cannot have an associated K3 surface because such
a K3 would have ``negative point counts'' as well as certain special
cubic fourfolds that are not expected to have associated $K3$ surfaces
over $\overline{\F}_2$, yet whose nonspecial Weil polynomial is still
contained on Kedlaya and Sutherland's list, raising further questions
about associated $K3$ surfaces over finite fields (see the forthcoming
work of the first and third authors~\cite{auel_petok}).
Future census projects could help further populate the list of Weil
polynomials that are realized by $K3$ surfaces over $\F_2$.  

Finally, as stated in Theorem~\ref{thm:smooth}, our census also
provides a count of the $\F_2$-points of the complement of the generic
discriminant of cubic forms in six variables.  From this, we find that
the probability that a random cubic fourfold is smooth is about
$29\%$, and we connect this to asymptotic results of
Poonen~\cite{poonen},
Church--Ellenberg--Farb~\cite{church_ellenberg_farb},
Vakil--Wood~\cite{vakil_wood}, and Howe~\cite{howe} in algebraic
geometry, number theory, and topology.

\smallskip

This article is organized as follows.  In ~\S\ref{sec:filtrations}, we
present our new method for computing orbit representatives for a
finite group $G$ acting on a finite vector space $V$ admitting a
filtration by $G$-stable subspaces, which we coin the ``filtration
method.'' We also compare the computational complexity of our method
compared with that of more traditional
methods. In~\S\ref{sec:countinghypersurfaces}, we describe the range
of applicability of the filtration method to enumerating degree $d$
hypersurfaces in $\P^n$ over $\F_q$, including the case of cubic
fourfolds over $\F_2$. Finally, in~\S\ref{sec:surveyingthedata}
and~\S\ref{sec:zeta}, we compute many invariants associated to cubic
fourfold, including their counts of points, lines, and planes, their
automorphism groups, and their zeta functions. We also discuss many
connections and complements to the existing literature.

\medskip

We would like to thank Institute for Computational and Experimental
Research in Mathematics (ICERM) for supporting the workshop
\textit{Birational Geometry and Arithmetic} in May 2018, and the
Simons Foundation for supporting the \textit{Simons Collaboration on
Arithmetic Geometry, Number Theory, and Computation Annual Meeting} in
January 2023, where parts of this projects were worked on. We would
also like to thank Nick Addington, Eran Assaf, Fran\c{c}ois Charles,
Edgar Costa, Alexander Duncan, Sarah Frei, Connor Halleck-Dub\'e,
Brendan Hassett, Kiran Kedlaya, Tony V\'arilly-Alvarado, John Voight,
and David Zureick-Brown for helpful discussions along the way.  We are
especially grateful to Drew Sutherland for his help with using the
\texttt{pclmulqdq} instruction, and to Keerthi Madapusi Pera for
explaining the key ideas needed for the Tate conjecture for cubic
fourfolds in characteristic $2$. The first author received partial
support from Simons Foundation grant 712097, National Science
Foundation grant DMS-2200845, and a Walter and Constance Burke Award
and a Senior Faculty Grant from Dartmouth College; the second author
received support from the Simons Collaboration on Arithmetic Geometry,
Number Theory, and Computation grant 550029; the fourth author
received support from the Neukom Institute for Computational Science
at Dartmouth College.

\section{Orbits via filtrations}
\label{sec:filtrations}

Let $k$ be a finite field and $V$ be a finite-dimensional $k$-vector
space on which a finite group $G$ acts linearly and faithfully.  For
$v \in V$ we denote by $G.v$ the $G$-orbit containing $v$, and by
$G_v$ the stabilizer subgroup of $v$.  If we are only interested in
the cardinality of the orbit set $V/G$, then we can use the orbit
counting formula (sometimes called ``Burnside's Lemma'' or the
``Cauchy--Frobenius Lemma'')
\begin{equation} 
\label{eqn:burnside}
\abs{V/G} = \frac{1}{\abs{G}}\sum_{g \in G} \abs{V^g} = \frac{\abs{C}}{\abs{G}} \sum_{c
\in C} \abs{E_1(c)}
\end{equation}
where $C$ is the set of conjugacy classes in $G$ and $E_1(c)$ is the
$1$-eigenspace of a representative of $c \in C$, whose cardinality does
not depend on the representative.  

However, to assemble a list of orbit representatives, one has to work
harder.  To do this, one typically runs an algorithm, such as
union-find (see \S\ref{ss:complexity} for more details), to sort each
element of $V$ into orbits under $G$, as is done for quartic surfaces
over $\F_2$ in~\cite{kedlaya_sutherland-census}.  Alternatively, one
could develop a sufficiently good $G$-invariant hash function on $V$
and try randomly sampling elements of $V$ until one finds elements in
all the orbits.  The random sampling method will often succeed in
identifying elements in all large orbits after sampling
$O\bigl(\abs{V/G}\log(\abs{V/G})\bigr)$ elements, but it can fail to
find elements in small orbits in reasonable time.  This method has been
used successfully by Costa, Harvey, and Kedlaya~\cite{costa_talk} to
give a census of quartic $K3$ surfaces over $\F_3$ and was used by
Halleck-Dub\'e~\cite{halleck} to enumerate a set of orbit
representatives for 99.9\% of the cubic fourfolds over $\F_2$.

However, working directly on $V$ may prove to be too costly (as in the
case of cubic fourfolds over $\F_2$), and we introduce a method that
can avoid this.

\subsection{Filtration method}

Suppose that there is a filtration of the $G$-module $V$
\[ 
0 = W_0 \subset \dots \subset W_\ell \subset V
\] 
by $G$-submodules $W_i \subset V$, such that enumerating $G$-orbits in
all of the associated graded pieces $ W_{i+1}/W_i$ becomes a feasible
task. Then using such a filtration, we are able to compute a full set
of $G$-orbit representatives for $V$ by chasing lifts of $G$-orbits of
$V/W_\ell$ up the successive quotients
\[
  V \to V/W_1 \to \cdots \to  V/W_\ell.
\]

Let us illustrate the method with a single-step filtration
\[
  0 \subset W \subset V
\]
and consider $U = V/W$ with $G$-equivariant quotient map $\pi : V \to
U$.  We first note that every $G$-orbit $G.v$ in $V/G$ maps to a
$G$-orbit $G.\pi(v)$ in $U/G$. This lets us write the orbit space
$V/G$ of $G$ as a disjoint union

\[
  V/G = \bigsqcup_{O \in U/G} \pi\inv(O)/G
\]
over the orbit sets of $G$ acting on the inverse
images $\pi\inv(O) \subset V$ of orbits $O \in U/G$.

So to give representatives of the orbits of $V/G$, we need to give a complete list of representatives for the $G$-orbits in each of the component orbit spaces $\pi\inv(O)/G$. 
The following elementary lemma says that we can find a complete set of representatives for $\pi\inv(O)/G$ by considering the action of a smaller group on a smaller subspace.

\begin{lemma}\label{lemma:stabilizers}
Let $V$ be a $k$-vector space with a $G$-action and let $W$ be a
$G$-invariant subspace.  Let $v \in V$ and let $\pi\colon V \to V/W$ be the natural projection. Then:
  
  \begin{enumerate}
  \item \label{lemma:stabilizers1}
    The coset $\pi^{-1}(\pi(v)) = v+W$ has a natural $G_{\pi(v)}$-action, i.e., $v+W$ is a $G_{\pi(v)}$-set.
    
  \item  \label{lemma:stabilizers2}
    Let $O=G. \pi(v) \subset V/W$ denote the $G$-orbit of $\pi(v)$. Then the map
    \[
      (v+W)/G_{\pi(v)} \to \pi\inv(O)/G
    \]
    defined by $G_{\pi(v)}.x \mapsto G.x$, is a bijection.
    
  \item\label{lemma:stabilizers3}
    $G_v \le G_{\pi(v)}.$
  \end{enumerate}
\end{lemma}

The upshot of Lemma~\ref{lemma:stabilizers}(\ref{lemma:stabilizers2})
is that a complete set of orbit representatives for the orbit space
$(v+W)/G_{\pi(v)}$ will also be a complete set of orbit
representatives for $G$ acting on $\pi\inv(O)/G$, and it is a far less
expensive a computation to find the former set of representatives. We
now give the full description of an algorithm that uses this principle
across successive quotients to find a set of representatives of $V/G$.

In the following, if we have vector spaces $V \supseteq W$, we denote the quotient by $\pi_{V/W} \colon V \to V/W$.

\begin{algorithm} 
\label{algorithm:main} 
\texttt{Orbits}($G$, $X$, $\mathcal{F}$) \ \\
  \textbf{Input:}
  \begin{itemize}
    \setlength{\itemsep}{0pt}
    \item 
      A $k$-vector space $V$ with the action of a group $G$.
    \item
      A known $G$-invariant filtration
      $\mathcal{F}\: 0 = W_0 \subset \dots \subset W_\ell \subset V$
      of length $\ell$.

    \item
      A $G$-invariant affine subspace $X$ of $V$ such that $X + W_\ell = X$.
      
  \end{itemize}
  \textbf{Output:} A complete set of orbit representatives for $G$ acting on $X$, together with their stabilizers.

\noindent\textbf{Steps:}

  \begin{enumerate}

 \item\label{algo:step-1}
    If $\ell = 0$ then \textbf{return}  \texttt{Orbits}($G$, $X$), a set of orbit representatives of $G$ acting on $X$ together with their stabilizers.

    \item
      Set $\overline{\mathcal{F}}: 0 = W_1/W_1  \subset \dots  \subset
      W_\ell/W_1 \subset V/W_1$, a  $G$-invariant filtration of length
      $\ell-1$, i.e., $\overline{\mathcal{F}}$ is the reduction of
      $\mathcal{F}$ modulo $W_1$.

    \item
      Compute \texttt{Orbits}($G$, $X/W_1$, $\overline{\mathcal{F}}$) via recursion.

 \item\label{algo:step-4}
    For each orbit representative $y \in X/W_1$ with stabilizer $G_y$ found in the previous step, compute \texttt{Orbits}($G_y$, $\pi_{V/W_1}^{-1}(y)$) together with their stabilizers.

  \item
    \textbf{return} the union of results from step~(\ref{algo:step-4}).
    
  \end{enumerate}

\end{algorithm}

The orbit computations in Steps~(\ref{algo:step-1})~and~(\ref{algo:step-4}) above can be computed by a generic
algorithm. Our implementation uses the default methods in \texttt{Magma}~\cite{Magma}.

\subsection{Complexity comparison}
\label{ss:complexity}

We consider the situation of a finite group $G$ acting on a finite set
$X$ and define the \textit{expected stabilizer order}
$$
e_G(X) = \frac{1}{\abs{X}}\sum_{x \in X} \abs{G_x}
$$
of the $G$-set $X$.  Then $1 \leq e_G(X) \leq \abs{G}$ with $e_G(X)=1$ if
and only if $G$ acts freely on $X$ and $e_G(X) = \abs{G}$ if and only if
$G$ acts trivially on $X$.  We remark that the proof of the orbit
counting formula implies that
$$
\abs{X/G} = e_G(X) \frac{\abs{X}}{\abs{G}}. 
$$

A naive orbit partition algorithm, such as implemented in union-find,
to partition all the elements of $X$ into orbits, works by iteratively
selecting the next unlabelled element $x \in X$ and then by labelling
the elements of $G . x$ as being in the same orbit by enumerating over
$G$. One easily finds the runtime of this procedure.

\begin{lemma}
\label{lem:runtime_naive}
The runtime of a naive orbit partition algorithm to partition the
elements of $X$ into orbits under $G$ is proportional to
$$
\abs{G} \cdot \abs{X/G} = e_G(X) \cdot \abs{X}. 
$$
\end{lemma}

In the situation we are interested in, where $X$ is a vector space
with a faithful $G$-representation, we usually have that $e_G(X)$ is
approximately equal to~$1$. For example, the expected order of the
stabilizer of a cubic fourfold over $\mathbb{F}_2$ turns out to be
approximately~$1.04$.  

The main improvement introduced by using the filtration method is that
one runs several orbit partitions over linear spaces of smaller
dimension. We give a precise estimate of the improvement in runtime.

\begin{lemma}
  \label{lem:runtime}
  For a single-step filtration $0\subset{W}\subset{V}$ of $G$-modules,
  with $U = V/W$, the runtime complexity of
  Algorithm~\ref{algorithm:main} is proportional to

  \[
    \frac{e_G(W) e_G(U)}{e_G(V)} \cdot \abs{V/G}.
  \]

\end{lemma}

\begin{proof}
Let $\texttt{Orbits}(G,U)$ be a chosen set of representatives $v \in
V$ for the $G$-orbits on $U$.  For a vector $v \in V$ denote its
image in $U$ by $\bar v$.  The filtration method algorithm consists
of a union find over $(v+W)/G_{\overline{v}}$ for each $v \in
\texttt{Orbits}(G,U)$. The runtime complexity is then proportional
to
  \begin{equation}
    \label{eq:runtime-actual}
    \sum_{v \in \texttt{Orbits}(G,U)} \abs{G_{\overline{v}} }\cdot \abs{(v+W)/G_{\overline{v}}}.
  \end{equation}
  One sees immediately that this is bounded by $\abs{G} \cdot \abs{V/G} =
  e_G(V) \cdot \abs{V}$, which is the runtime of a naive orbit partition
  algorithm, see Lemma~\ref{lem:runtime_naive}. On the other hand, for
  each $v$ we have
  \[
    \abs{G_{\overline{v}}} \cdot \abs{(v+W)/G_{\overline{v}}} = \sum_{g \in G_{\bar v}} \abs{(v+W)^g} \leq \sum_{g \in G} \abs{(v+W)^g}
    \leq \sum_{g \in G} \abs{W^g} = \abs{G} \cdot \abs{W/G}
  \]
  where the rightmost inequality follows from the observation that 
  for any element $z \in (v+W)^g$, translation by $z$ induces a
  bijection $W^g \to (v+W)^g$.  Thus \eqref{eq:runtime-actual} is
  bounded by $\abs{W/G} \cdot \abs{U/G} \cdot \abs{G} = {e_G(W) e_G(U)}/{e_G(V)} \cdot \abs{V/G}$.
\end{proof}

Lemma~\ref{lem:runtime} shows that the filtration method, in the
presence of a nontrivial filtration, will strictly improve upon
(unless the action is trivial) a naive orbit partition for the
purposes of finding a set of orbit representatives.

If we make the heuristic assumption that the expected stabilizer
orders of $W,$ $U,$ and $V$ are all approximately equal to~$1$, then the
asymptotic runtime is linear in the total number of orbits in $V$,
i.e., linear in the size of the output. In particular, our algorithm
would be asymptotically optimal.  This assumption seems to hold in the cases
identified in \S\ref{sec:otherfiltration}.  

\section{Enumerating hypersurfaces}
\label{sec:countinghypersurfaces}
Let $n, d$ be positive integers, and $\F_q$ denote the finite field
with $q$ elements. In this section, we explain when and how the
filtration method (Algorithm~\ref{algorithm:main}) can be used to
produce a complete enumeration of the set of $\F_q$-isomorphism
classes of degree $d$ hypersurfaces in $\P_{\F_q}^{n+1}$.

\begin{prop}
\label{prop:hypersurfaceautsF2} 
Let $k$ be any field.  If $n \ge 3$ and $d \ge 3$, then the set of
$k$-isomorphism classes of degree $d$ hypersurfaces in $\P_k^{n+1}$ is
in bijection with the set of $\PGL_{n+2}(k)$-orbits on the set of lines
$\P(\Sym^d(k^{n+2}))(k)$ in $\Sym^d(k^{n+2})$.

Moreover, if every element of $k^\times$ is a $d$th power, e.g., if $k
= \F_q$ and $q-1$ is relatively prime to $d$, then the set of
$k$-isomorphism classes of degree $d$ hypersurfaces in $\P_k^{n+1}$ is
in bijection with the set of nonzero $\GL_{n+2}(k)$-orbits on the
$k$-vector space $\Sym^d(k^{n+2})$.
\end{prop}
\begin{proof}
The Grothendieck--Lefschetz
Theorem~\cite[Exp.~XII,~Corollaire~3.6]{sga-2} says that any
automorphism of a (not necessarily smooth) hypersurface of dimension
$n\ge 3$ and degree $d \ge 3$ extends to the ambient $\P_k^{n+1}$.  In
particular, two such hypersurfaces are $k$-isomorphic if and only if
they lie in the same $\PGL_{n+2}(k)$-orbit of the linear system of
$\mathcal{O}_{\P^{n+1}}(d)$.  

When every element of $k^\times$ is a $d$th power, the central $\Gm
\subset \GL_{n+2}$ acts transitively on the set of multiples of a
given homogeneous form of degree $d$ over $k$, so that the natural
surjective map
$$
\Sym^d(k^{n+2})/\GL_{n+2}(k) \to \P(\Sym^d(k^{n+2}))(k)/\PGL_{n+2}(k)
$$ 
is a bijection.
\end{proof}

\subsection{A filtration on cubic fourfolds over $\F_2$}
By Proposition~\ref{prop:hypersurfaceautsF2}, the
$\GL_{6}(\F_2)$-orbits of nonzero cubic forms in $6$ variables are
precisely the $\F_2$-isomorphism classes of cubic fourfolds. We will
now show how the filtration method lets us enumerate a representative
cubic fourfold in each isomorphism class, equivalently, a complete set
of nonzero $\GL_6(\F_2)$-orbit representatives on
$V=H^0(\mathbb{P}^5_{\F_2}, \mathcal{O}(3))$.

Using Equation (\ref{eqn:burnside}), one computes that that the number
of orbits is $3\,718\,650$, which seems manageable compared the total
number $\abs{V} = 2^{56}$ of cubics.  Even with the unrealistically
generous assumption that computing $f^g$, for some general $f \in V$
and $g \in \GL_6(\F_2)$, takes $10^{-9}$ (s), a naive orbit partition
algorithm applied to $V$ using a single 4 GHz processor would take at
least
\[
  \frac{2^{56}}{86400 \cdot 4 \cdot 10^9} \sim 208 \text{ days}.
\]
In other words, a direct, parallelized, and highly optimized
implementation of union-find might enumerate all of the orbits, but it
would nevertheless take a while.

Instead, we make use of a natural two-step filtration of
$G$-submodules 
\[
0 \subset W_1 \subset W_2 \subset V
\]
and employ
Algorithm~\ref{algorithm:main}.  Here, $W_1 \subset V$ is the subspace
of {\it Waring representable} cubic forms, i.e., those that can be
written as a sum of cubes of linear forms, and $W_2 \subset V$ is the
subspace of cubic forms that can be written as a sum of products of a
linear form and a square of a linear form.  In other words,
\begin{align*}
  W_1 &= \text{span}\{l^3 \; : \; l \in
        H^0(\mathbb{P}^5_{\F_2}, \mathcal{O}(1))\}\\%, &\dim(W_1) = 21, \\
  W_2 &= \text{span}\{l_1 \cdot l_2^2 \; : \; l_1, l_2 \in
        H^0(\mathbb{P}^5_{\F_2}, \mathcal{O}(1))\}.%, &\dim (W_2) = 36.
\end{align*}
A computer calculation shows that $\dim_{\F_2}(W_1) =
21$ and $\dim_{\F_2}(W_2) = 36$.

Our implementation in \texttt{Magma}~\cite{Magma} of
Algorithm~\ref{algorithm:main} with this particular two-step
filtration outputs a complete set of representatives for the
$3\,718\,650$ orbits in $V$ with runtime under $100$ minutes on a
household laptop computer. %\footnote{Macbook Pro with M1 processor}.
A complete set of orbit representatives together with the code to read them in is available as an ancillary file in the arXiv distribution of this article. Our complete code library is available from~\cite{github}. The intrinsic {\tt LoadCubicOrbitData} from the {\tt CubicLib.m} library reads in the orbit representatives.

\subsection{A filtration on quartic surfaces over $\F_2$}
As an alternative to the orbit partition method employed
in~\cite{kedlaya_sutherland-census}, we implemented the filtration
method to find a representative of each $\PGL_4(\F_2)$-orbit of
quartic surfaces over $\F_2$.  We note that because there are
automorphisms of $K3$ surfaces that do not fix a given degree $4$
polarization, some isomorphism classes split up into different linear
orbits, but a list of orbit representatives will contain a complete
set of isomorphism classes as a subset.  By
Lemma~\ref{prop:hypersurfaceautsF2}, we can compute the
$\GL_4(\F_2)$-orbits on the $35$-dimensional $\F_2$-vector space $V =
H^0(\P^3_{\F_2},\mathcal{O}(1))$ of homogeneous quartic forms in four
variables. Let $W \subset V$ be the submodule spanned by all
quartics of the form $l_1^3 l_2 + l_1 l_2^3$ where $l_1$ and $l_2$ are
linear forms. A computer calculation shows that $\dim_{\F_2}(W)=20$.
Then Algorithm~\ref{algorithm:main}, with the one-step filtration
$0\subset W \subset V$, finds a complete set of orbit
representatives for the $1\,732\,564$ orbit in $V$ in about $3$ minutes
on a laptop computer. We stumbled upon the $G$-submodule $W\subset
V$ using \texttt{Magma}'s {\tt IsIrreducible} intrinsic.

\subsection{Cubic fourfolds over $\F_3$}

Unfortunately, the $\GL_6(\F_3)$-module of the $56$-dimensional
$\F_3$-vector space $V$ of cubic forms in six variables has an
irreducible composition factor of dimension $50$.  Hence the
filtration method alone does not provide a sufficiently significant
speedup to make the orbit enumeration problem computational tractable
in this case.

\subsection{Filtration method for general hypersurfaces} \label{sec:otherfiltration}

One may wonder about the extent to which the filtration method
presented in \S\ref{sec:filtrations} can aid in the census of
isomorphism classes of hypersurfaces of dimension $n$ and degree $d$
over $\mathbb{F}_q$ for various $(n, d, q)$. Since the $d$th symmetric
power of the standard representation of the linear algebraic group
$\GL_{n+2}$ is irreducible, the $\GL_{n+2}(\F_q)$-representation
$\Sym^d(\F_q^{n+2})$ is irreducible for all $\F_q$ with
$\mathrm{char}(\F_q) > d$ (cf.\ \cite[Theorem~1.1]{steinberg}), hence
the filtration method is not applicable.  On the other hand, the
$\GL_{n+2}(\F_q)$-representation $\Sym^d(\F_q^{n+2})$ is reducible for
all $\F_q$ with $\mathrm{char}(\F_q) \leq d$, hence the filtration
method does offer some speedup when~$q$ is small.  However, one
quickly sees (see Tables~\ref{fig:filtration-feasibility1},
\ref{fig:filtration-feasibility2}, \ref{fig:filtration-feasibility3})
even for moderately small parameters that such a census is infeasible
simply because it is not computationally practical to store the
answer. (We consider $10^{14}$ a generous allowance for the maximum
number of orbits that can be computed.)

When the total number of orbits is reasonably sized, the determination
of whether a particular value of $(n,d,q)$ is in the feasible range is
based on timings for computing $g . x$ on a standard laptop. We also
assume that 100 cores are available to parallelize the computation
over the course of a year. Thus, the projections listed in
Tables~\ref{fig:filtration-feasibility1},
\ref{fig:filtration-feasibility2}, \ref{fig:filtration-feasibility3}
are perhaps excessively optimistic.

\begin{figure}[h]
  \centering
  {\tiny \begin{tabular}{c|ccccccccccccccccc}
    ~\;~$n \backslash d$ & 2 & 3 & 4 & 5 & 6 & 7 & 8 & 9 & $\cdots$ & 48 & 49 \\ \hline
    0 & Y$|$Y & Y$|$Y & Y$|$Y & Y$|$Y & Y$|$Y & Y$|$Y & Y$|$Y & Y$|$Y & $\cdots$ & Y$|$Y & X \\
    1 & Y$|$Y & Y$|$Y & Y$|$Y & Y$|$Y & Y$|$Y & Y$|$Y & Y$|$Y & X & $\cdots$ & X & X \\
    2 & Y$|$Y & Y$|$Y & Y$|$Y & \fbox{N$|$Y} & X & X & X & X & $\cdots$ & X & X \\
    3 & Y$|$Y & Y$|$Y & X & X & X & X & X & X & $\cdots$ & X & X \\
    4 & Y$|$Y & \fbox{N$|$Y} & X & X & X & X & X & X & $\cdots$ & X & X \\
    5 & Y$|$Y & \fbox{N$|$Y} & X & X & X & X & X & X & $\cdots$ & X & X \\
    6 & Y$|$Y & X & X & X & X & X & X & X & $\cdots$ & X & X \\
    7 & Y$|$Y & X & X & X & X & X & X & X & $\cdots$ & X & X \\
    8 & \fbox{N$|$Y} & X & X & X & X & X & X & X & $\cdots$ & X & X \\
    $\vdots$ & $\vdots$ & $\vdots$ & $\vdots$ & $\vdots$ & $\vdots$ & $\vdots$
             & $\vdots$ & $\vdots$ & $\vdots$ & $\vdots$ & $\vdots$ \\
   20 & \fbox{N$|$Y} & X & X & X & X & X & X & X & $\cdots$ & X & X \\
  \end{tabular}}
  \caption{List of feasible cases for $q=2$. Too many orbits indicated
  by X; and Y, N indicate yes, no, respectively. The left symbol is
  for basic union-find, and the right symbol is using the filtration
  method.}
  \label{fig:filtration-feasibility1}
\end{figure}

\begin{figure}[h]
  \centering
 {\tiny \begin{tabular}{c|ccccccccccccccccc}
    ~\;~$n \backslash d$ & 2 & 3 & 4 & 5 & 6 & 7 & 8 & 9 & $\cdots$ & 31 & 32 \\ \hline
    0 & Y$|$Y & Y$|$Y & Y$|$Y & Y$|$Y & Y$|$Y & Y$|$Y & Y$|$Y & Y$|$Y & $\cdots$ & Y$|$Y & X \\
    1 & Y$|$Y & Y$|$Y & Y$|$Y & Y$|$Y & Y$|$Y & \fbox{N$|$Y} & X & X & $\cdots$ & X & X \\
    2 & Y$|$Y & Y$|$Y & \fbox{N$|$Y} & X & X & X & X & X & $\cdots$ & X & X \\
    3 & Y$|$Y & \fbox{N$|$Y} & X & X & X & X & X & X & $\cdots$ & X & X \\
    4 & Y$|$Y & \fbox{N$|$Y} & X & X & X & X & X & X & $\cdots$ & X & X \\
    5 & Y$|$Y & X & X & X & X & X & X & X & $\cdots$ & X & X \\
    6 & N$|$N & X & X & X & X & X & X & X & $\cdots$ & X & X \\
  \end{tabular}}
  \caption{List of feasible cases for $q=3$. Too many orbits indicated
  by X; and Y, N indicate yes, no, respectively. The left symbol is
  for basic union-find, and the right symbol is using the filtration
  method.}
  \label{fig:filtration-feasibility2}
\end{figure}

\begin{figure}[h]
  \centering
  {\tiny \begin{tabular}{c|ccccccccccccccccc}
    ~\;~$n \backslash d$ & 2 & 3 & 4 & 5 & 6 & 7 & 8 & 9 & $\cdots$ & 22 & 23 \\ \hline
    0 & Y$|$Y & Y$|$Y & Y$|$Y & Y$|$Y & Y$|$Y & Y$|$Y & Y$|$Y & Y$|$Y & $\cdots$ & Y$|$Y & X \\
    1 & Y$|$Y & Y$|$Y & Y$|$Y & Y$|$Y & \fbox{N$|$Y} & X & X & X & $\cdots$ & X & X \\
    2 & Y$|$Y & Y$|$Y & X & X & X & X & X & X & $\cdots$ & X & X \\
    3 & Y$|$Y & N$|$N & X & X & X & X & X & X & $\cdots$ & X & X \\
    4 & Y$|$Y & X & X & X & X & X & X & X & $\cdots$ & X & X \\
    5 & N$|$N & X & X & X & X & X & X & X & $\cdots$ & X & X \\
  \end{tabular}}
  \caption{List of feasible cases for $q=5$. Too many orbits indicated
  by X; and Y, N indicate yes, no, respectively. The left symbol is
  for basic union-find, and the right symbol is using the filtration
  method.}
  \label{fig:filtration-feasibility3}
\end{figure}

%%% SECTION
\section{Surveying the database}\label{sec:surveyingthedata}

Our census lets us survey several interesting invariants associated to
cubic fourfolds over $\F_2$. We now outline the main invariants
studied in this section. Our dataset includes the $\F_2$-automorphism
groups of every cubic fourfold over $\F_2$, which happens to coincide
with $\GL_6(\F_2)$-stabilizer that we compute in the course of running
the filtration method. Some highlights of this automorphism data are
presented in~\S\ref{ss:auts}. In~\S\ref{ss:countingcubics}, the orders
of the automorphism groups are used to count $\F_2$-points on the {\it
discriminant complement} of the moduli space of cubic hypersurfaces in
$\P^5$; we relate this count to work of Poonen~\cite{poonen},
Church--Ellenberg--Farb~\cite{church_ellenberg_farb},
Vakil--Wood~\cite{vakil_wood}, and Howe~\cite{howe}. Finally,
in~\S\ref{ss:linearspaces}, we compute all $\F_2$-lines and planes on
cubics in our database, verifying statistics predicted by
Debarre--Laface--Roulleau~\cite{debarre_laface_roulleau} and giving a
lower bound on the number of smooth {\it rational} cubic fourfolds
over $\F_2$.

\subsection{Automorphisms of cubics}\label{ss:auts} 
The automorphism groups of cubic hypersurfaces over various fields
have been well-studied.  Over the complex numbers, the {\it
symplectic} automorphism groups of smooth cubic fourfolds were
recently fully classified by Laza and Zheng~\cite{laza-zheng}, and
they additionally prove that the Fermat cubic has the largest
automorphism group of any smooth cubic fourfold over $\C$ (see also
\cite{ouchi}).  In positive characteristic, we do not know of any body
of literature on the automorphism groups of cubic hypersurfaces of
dimension $>2$. The automorphism groups of cubic {\it surfaces} over
algebraically closed fields of {\it arbitrary characteristic} were
completely classified by Dolgachev and
Duncan~\cite{dolgachev_duncan}. Our orbit-finding method yields a
complete classification of the $\F_2$-automorphism groups of cubic
surfaces, threefolds, and fourfolds. We report some specific results
on cubic fourfolds here.

First, we compare the automorphism group of a hypersurface with its
stabilizer subgroup.  

\begin{prop}
\label{prop:automorphisms}
Let $k$ be a field and assume that $n \geq 3$ and $d \geq 3$.  Let $f
\in \Sym^d(k^{n+2})$ be nonzero and $X \subset \P_{k}^{n+1}$ be the
associated projective hypersurface.  Then the
$\PGL_{n+2}(k)$-stabilizer of the line spanned by $f$ is isomorphic to
the group $\Aut_k(X)$ of $k$-automorphisms $X$.

Moreover, if every element in $k^\times$ is a $d$th power (e.g., if
$k=\F_q$ and $q-1$ is relatively prime to $d$) then the
$\GL_{n+2}(k)$-stabilizer of $f$ is isomorphic to the group
$\Aut_k(X)$ of $k$-automorphisms $X$.
\end{prop}

\begin{proof}
The first statement follows immediately from
Proposition~\ref{prop:hypersurfaceautsF2}. As for the second
statement, if $G_f \subset \GL_{n+2}$ denotes the stabilizer $k$-subgroup
scheme of $f$, then we have a short exact sequence of $k$-group schemes
\[
1 \to \mu_d \to G_f \to \Aut_k(X) \to 1,
\]
where here, $\Aut_k(X)$ is considered as a constant group scheme.  The
associated exact sequence in flat cohomology starts
\[
1 \to \mu_d(k) \to G_f(k) \to \Aut(X) \to H^1(k, \mu_d).
\]
Under the hypotheses that every element in $k^\times$ is a $d$th
power, we have that $\mu_d(k)$ and $H^1(k,\mu_d)$ are trivial.  The
statement then follows since $G_f(k)$ coincides with the
$\GL_{n+2}(k)$-stabilizer subgroup of $f$.  Indeed, the exact
sequence in flat cohomology associated to the stabilizer subgroup
group scheme starts
\[
1 \to G_f(k) \to \GL_{n+2}(k) \to (\GL_{n+2}.f)(k) \to H^1(k,G_f)
\]
and, since $f$ is a $k$-rational point in the orbit, $G_f(k)$ is
precisely the subset of elements of $\GL_{n+2}(k)$
acting trivially on $f$. 
\end{proof}

Hence in the case of cubic fourfolds over $\F_2$, the stabilizer of a
cubic form coincides with the $\F_2$-automorphism group of its
associated hypersurface.  To compute the $G=\GL_6(\F_2)$-stabilizer of
the cubic form $f\in V$ that defines $X$, we use the same idea behind
Algorithm~\ref{algorithm:main} to successively reduce the order of the
acting group: letting $\pi_1 \colon V \to V/W_1$ and $\pi_2 \colon
V/W_1 \to V/W_2$ denote the natural projections,
Lemma~\ref{lemma:stabilizers}(\ref{lemma:stabilizers3}) tells us that
\[G_f = (G_{\pi_1(f)})_f = ((G_{\pi_2(f)})_{\pi_1(f)})_f.\]

\begin{remark}[The isomorphism problem] 
A similar application of the filtration method gives an efficient
algorithm for solving the cubic isomorphism problem over~$\F_2$: our
library includes an intrinsic {\tt IsEquivalentCubics} which
determines whether two cubic fourfolds are $\F_2$-isomorphic, and if
they are, returns an explicit isomorphism between them. It runs in
about 0.2 seconds per pair of cubics on a household laptop.
\end{remark}

We now discuss some highlights emerging from our computation of the
automorphism groups of cubic fourfolds over $\F_2$. It is a well-known
result that the automorphism group of a generic hypersurface of
dimension $n \geq 2$ and degree $d \geq 3$, over an algebraically
closed field, is trivial (see~\cite{matsumura_monsky}). Indeed, our
survey shows that $\Aut_{\F_2}(X)=\{\text{id} \}$ for most cubic
fourfolds.

\begin{comp}
\label{comp:mostlytrivial} 
Among the $3\,718\,649$ isomorphism classes of cubic fourfolds over
$\F_2$, there are $3\,455\,271$, or about $92.9\%$, with trivial
stabilizer. Among the  $1\, 069\, 562$ isomorphism classes of\ {\it smooth} cubic fourfolds over
$\F_2$, there are $1\,029\,478$, or about $96.3\%$, with trivial stabilizer.
\end{comp}

We summarize our computation of all of the nontrivial automorphism groups of cubic fourfolds in the next theorem.

\begin{theorem}
If $X$ is cubic fourfold over $\F_2$, then the order of
$\Aut_{\F_2}(X)$ is one of the following 87 possibilities:
  \[
    \numberlist{  
      1, 2, 3, 4, 5, 6, 7, 8, 9, 10, 12, 15, 16, 18, 20, 21, 24, 
      30, 32, 36, 42, 48, 60, 63, 64, 72, 84, 96, 108, 
      120, 126, 128, 144, 160, 168, 192, 256, 288, 320, 384, 
      512, 576, 648, 672, 720, 768, 882, 
      1024, 1152, 1344, 1440, 1536, 
      1920, 2016, 2048, 2160, 2304, 
      3072, 3840, 4032, 4608, 6144, 7680, 9216, 
      10752, 11520, 12288, 18432, 23040, 24576, 27648, 32256, 
      36864, 73728, 86016, 172032, 258048, 344064, 516096, 1105920, 1451520, 
      1806336, 5160960, 10321920, 15482880, 30965760, 319979520.
    }
      \\
  \\
  \]
  If $X$ is a smooth cubic fourfold over $\F_2$, then the order of
  $\Aut_{\F_2}(X)$ is one of the following 40 possibilities:
  \[
    \numberlist{1, 2, 3, 4, 5, 6, 7, 8, 9, 10, 12, 15, 16, 18, 24, 30, 32,
      36, 48, 64, 72, 96, 108, 120, 128, 144, 160, 192, 288, 384, 512, 576,
      648, 1440, 1536, 2160, 4608, 10752, 23040, 1451520.}
  \]
\\
\end{theorem}

%% New version
\begin{remark}[An extremal cubic fourfold]
\label{rem:DKHK}
The smooth cubic fourfold 
  \begin{equation} %\label{eqn:DKHK}
    X_1 : \ x_0x_3^2 +x_1x_4^2 + x_2x_5^2+x_0^2x_3 + x_1^2x_4 + x_2^2 x_5=0
  \end{equation}
  is the unique cubic fourfold over $\F_2$ with
  $\abs{\Aut(X_1)(\F_2)}=1\,451\,520$; in fact, our stabilizer computations
  yield that the $\F_2$-automorphisms form a group isomorphic to the
  symplectic group $\mathrm{Sp}(6, \F_2)$, and this is the largest
  $\F_2$-automorphism group of all smooth $\F_2$-cubic fourfolds. This
  cubic fourfold was also studied in~\cite{dolgachev_kondo,
  hulek_kloosterman}.

  The appearance of the symplectic group admits a simple explanation. We
  consider the $\F_4/\F_2$-Hermitian form defined by
  \[
    H(x, y) :=
    x_0y_5^2 + x_1y_4^2 + x_2 y_3^2 + x_3 y_4^2 + x_4 y_1^2 + x_5 y_0^2
  \]
  where $x = (x_0, \dotsc, x_5)$ and $y = (y_0, \dotsc, y_5)$ are in
  $\F_4^6$.  The map from Hermitian forms to cubic forms defined by
  $H(x, y) \mapsto H(x, x)$ is injective, so in particular, $g \in
  \GL_6(\F_4)$ fixes $H(x, y)$ if and only if it
  fixes $H(x, x)$. We observe that the group of
  $\F_2$-rational points of the unitary group of $H$ is isomorphic to
  $\operatorname{Sp}(6, \F_2)$.
  \end{remark}  

\begin{remark}[The Fermat cubic fourfold]
\label{rem:Fermat}  
There is also a unique smooth cubic fourfold $X_2$ with
  $\abs{\Aut(X_2)(\F_2)} = 23040$, the {\it Fermat cubic fourfold}
  \[
X_2 : \ x_0^3 + x_1^3 + x_2^3 + x_3^3 + x_4^3 + x_5^3=0.
\]
  Incidentally, there are also three singular cubic fourfolds with an
  automorphism group of this order, two of which have
  $\F_2$-automorphism group isomorphic to $\Aut_{\F_2}(X_2)$. One easily
  sees that the cubic $X_1$ is an $\F_2$-form of the Fermat cubic $X_2$,
  split by any $K/\F_2$ containing a primitive third root of unity; in
  particular, $X_1 \times_{\F_2} \F_4 \cong X_2 \times_{\F_2} \F_4$.
\end{remark}

\subsection{Point counting on moduli spaces}\label{ss:countingcubics}
Let $U_d$ be the the {\it discriminant complement} in the Hilbert
scheme of degree $d$ hypersurfaces over $\F_q$, i.e., $U_d$ is the
open subscheme of $\P_{\F_q}^{\binom{d+n}{d}}$ parametrizing smooth
degree $d$ hypersurfaces in $\P^n_{\F_q}$. An relevant question is:
what is the probability that a randomly chosen hypersurface is smooth?
In~\cite{poonen}, Poonen gave an answer asymptotically in $d$,
proving that

\[\lim_{d\to \infty} \frac{\abs{U_d(\F_q)}}{\abs{\P^{\binom{d+n}{d}}(\F_q)}} = \frac{1}{Z_{\P^n}(n+1)} = \prod_{1 \le k \le n }\left(1-\frac{1}{q^k}\right). \]

Poonen's result on point counting on discriminant complements has been
related to the phenomena of stabilization in the Grothendieck ring in
work of Vakil and Wood~\cite{vakil_wood}, and has also been studied
from the perspective of homological and representation stability by
Church, Ellenberg, and Farb~\cite{church_ellenberg_farb} and by
Howe~\cite{howe}.

 In the case of hypersurfaces in $\P^5$ over $\F_2$, Poonen's result says that the probability as $d \to \infty$ that a randomly chosen hypersurface of degree $d$ is smooth should be 
\begin{equation}\label{eqn:poonenprob}
 \prod_{1 \le k \le 5 }\left(1-\frac{1}{2^k} \right)= 0.298004150390625 
 \end{equation}
 Our computations let us compute $\abs{U_3(\F_2)}$ exactly, and therefore
 the probability that a randomly chosen {\it cubic} hypersurface in
 $\P_{\F_2}^5$ is smooth.

We first give some point counts on a series of related moduli spaces:
\begin{alignat*}{5}
  & \mathcal{C},  &&\text{ \ the coarse moduli space of cubic fourfolds}, \\ 
  &\mathcal{C}^{sm}, &&\text{ \ the coarse moduli space of smooth cubic fourfolds}, \\
  &\mathscr{C}, &&\text{ \ the moduli stack of cubic fourfolds, and}\\
  &\mathscr{C}^{sm}, &&\text{ \ the moduli stack of smooth cubic fourfolds}.
\end{alignat*}

The first point count can be computed without any difficulty using the Burnside formula~\eqref{eqn:burnside}
\[\abs{\mathcal{C}(\F_2)} = 3\,718\,649.\]
 
Next we compute the number of smooth cubic fourfolds. One cannot
determine the number of smooth orbits easily from Burnside's
lemma. Instead, we directly count the smooth obits using our database
of explicit orbit representatives. We can thus report that the number of
smooth cubic fourfolds up to isomorphism is
\[
\abs{\mathcal{C}^{sm}(\F_2)} = 1\,069\,562, 
\]
or about $28.76\%$.

Using our computations of the automorphism subgroups discussed above,
we determine the stacky point counts as well:
\[
  \abs{\mathscr{C}(\F_2)} =
  \sum_{X \in \mathcal{C}(\F_2)} \frac{1}{\abs{\Aut_{\F_2}(X)} } = \frac{4803839602528529}{1343913984} \approx 3\,574\,514.18746, 
\]
and
\[
 \abs{\mathscr{C}^{\sm}(\F_2) } =  \sum_{X \in \mathcal{C}^{sm}(\F_2)} \frac{1}{\abs{\Aut(X)}} = 1\,048\,581.
\]
It is interesting to note that the stacky point count of $\mathscr{C}^{\sm}$ is an integer.

Finally, summing the sizes of the orbits $[\PGL_6(\F_2) : \Aut_{\F_2}(X)]$ gives the $\F_2$-point count on the discriminant complement $U_3$.

\begin{theorem}\label{thm:discriminantcomplement}  
The cardinality of the set of smooth cubic fourfolds over $\F_2$ (not
considered up to isomorphism) is
\[
\abs{U_3(\F_2)} = 21\,138\,040\,038\,850\,560.
\] 
and thus the probability that a random cubic fourfold over $\F_2$ is smooth is exactly
\[
\frac{\abs{U_3(\F_2)}}{\abs{\P^{56}(\F_2)}}  =
0.29334923433225412736646831035614013671875.
\]

\end{theorem}

We compare this count to some other counts of the proportion of smooth
small degree hypersurfaces over $\F_2$, see Figure~\ref{fig:poonen-counts}.

\begin{figure}[h]
\renewcommand{\arraystretch}{1.5}
  {\tiny \begin{tabular}{c|ccccc|c}
    ~\;~$n \backslash d$ & 1 & 2 & 3 & 4 & 5  & $ \displaystyle\prod_{ 1 \le k \le n+1 }\left(1-\frac{1}{2^k} \right)$  \\ \hline
    1 & 1 & $\frac{4}{9} \approx 0.444$ & $\frac{112}{341} \approx 0.328$ &
    $\frac{1560}{4681} \approx 0.333$ & $\frac{98304}{299593} \approx 0.328$ &
    $\frac{3}{8} = 0.375$  \\
    2 & 1 &  $\frac{448}{1023}\approx0.438$ & $\frac{21504}{69905} \approx 0.308$
    & $\frac{10590854400}{34359738367} \approx 0.308$ & &  $\frac{21}{64} \approx 0.328$  \\
    3 & 1 & $\frac{64}{151} \approx 0.424$ & $\frac{330301440}{1108378657}
    \approx 0.298$ & &  &  $\frac{315}{1024} \approx 0.308$  \\
    4 & 1 & $\frac{126976}{299593} \approx 0.424$ &
    $\frac{1409202669256704}{4803839602528529} \approx 0.293$  &  &  &
    $\frac{9765}{32768} \approx 0.298$  \\
  \end{tabular} }
  \caption{List of proportion of hypersurfaces of degree $d$ in $\P^{n+1}$ that are smooth, over $\F_2$. The last column is the asymptotic proportion as $d \to \infty$ from Poonen's theorem. The proportion of cubic fourfolds which are smooth is closer to the Poonen limit than any other entry in the table.} 
  \label{fig:poonen-counts}
\end{figure}

\subsection{Lines on cubic fourfolds}\label{ss:linearspaces}
We compute the set of $\F_2$-lines on every cubic
fourfold. We find that there exist
exactly 65 cubic fourfolds which contain exactly one $\F_2$-line, only
29 of which are smooth (the first example of such a cubic was given
in~\cite{debarre_laface_roulleau}.) Our exhaustive computations confirm that every cubic
fourfold $X$ contains an odd number of $\F_2$-lines (so in particular they all contain at least one such line). In fact, this was already proved by Debarre--Laface--Roulleau.

\begin{lemma}[{\cite{debarre_laface_roulleau}}]
\label{lem:containsline}
  Every cubic fourfold over $\F_2$ contains a line defined over $\F_2$.
\end{lemma}
In~\cite{debarre_laface_roulleau}, this result is derived using a formula of Galkin and Shinder~(\cite[Equation 8]{debarre_laface_roulleau}, \cite{galkin_shinder}): there is a relation between point counts on $X$ 
and its Fano variety of lines $F_1(X)$, which in the case of cubic fourfolds over $\F_2$ yields
\begin{equation}
 \abs{F_1(X)(\F_2)} = \frac{\abs{X(\F_2)}^2 +2(1+2^4)\abs{X(\F_2)} +\abs{X(\F_4)}}{8} \ + \  4 \ \abs{\text{Sing}(X)(\F_2)}.
\end{equation}
In particular, since $\abs{X(\F_2)} \ge 1$ (by Chevalley--Warning), one has $\abs{F_1(X)(\F_2) }\ge 1$.

 If one is only interested in the {\it number} of lines on a cubic fourfold then the above formula suffices provided one computes point counts on $X$ (as we do in~\S\ref{sec:zeta} below), but we still computed the full set of $\F_2$-lines on each cubic with other applications in mind---for instance, if one is interested in searching for various {\it families} of lines on cubics, like planes and scrolls, it is useful to know the full set of lines.

 We plot a histogram, see Figure~\ref{fig:cubic-lines},
 of the count (weighted by stabilizer) of the number of isomorphism
 classes of cubics containing a given number of lines.  We also plot the same
 histogram for just the smooth cubics. The histograms match the
 prediction from~\cite{debarre_laface_roulleau}.

\begin{figure}[h]
  \centering
  \includegraphics[scale=0.3]{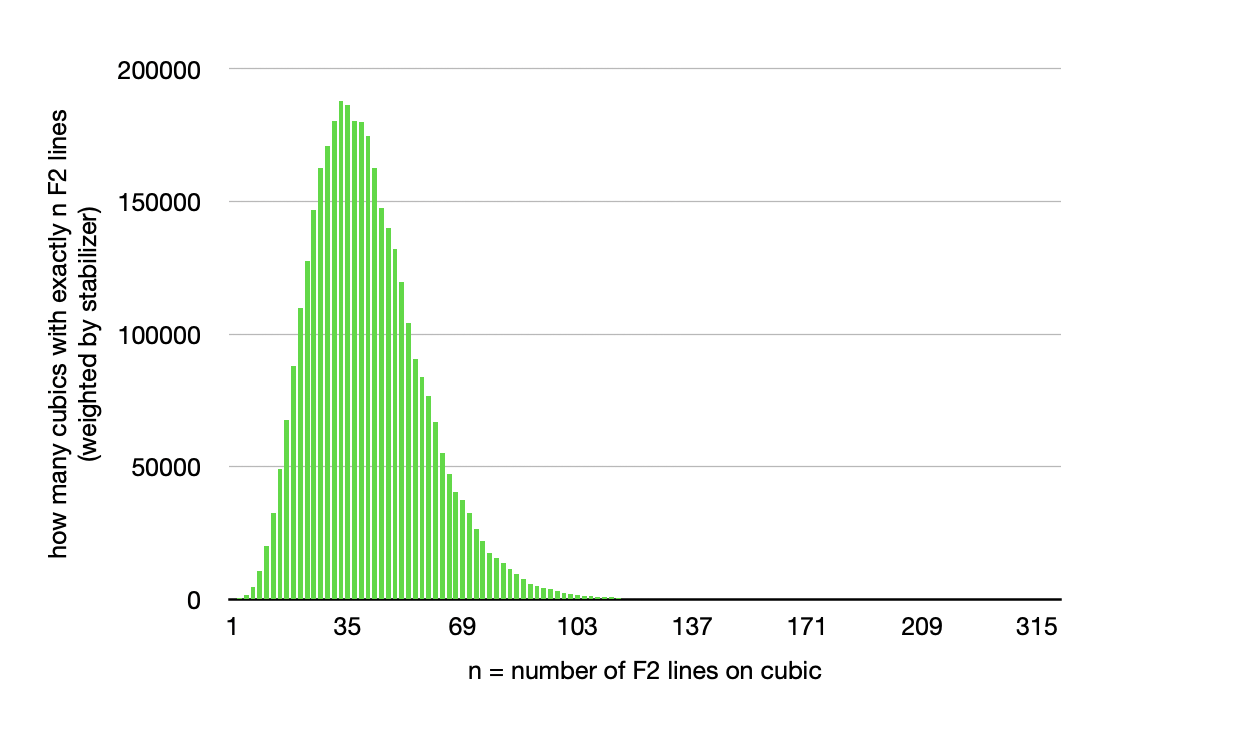}
  \includegraphics[scale=0.3]{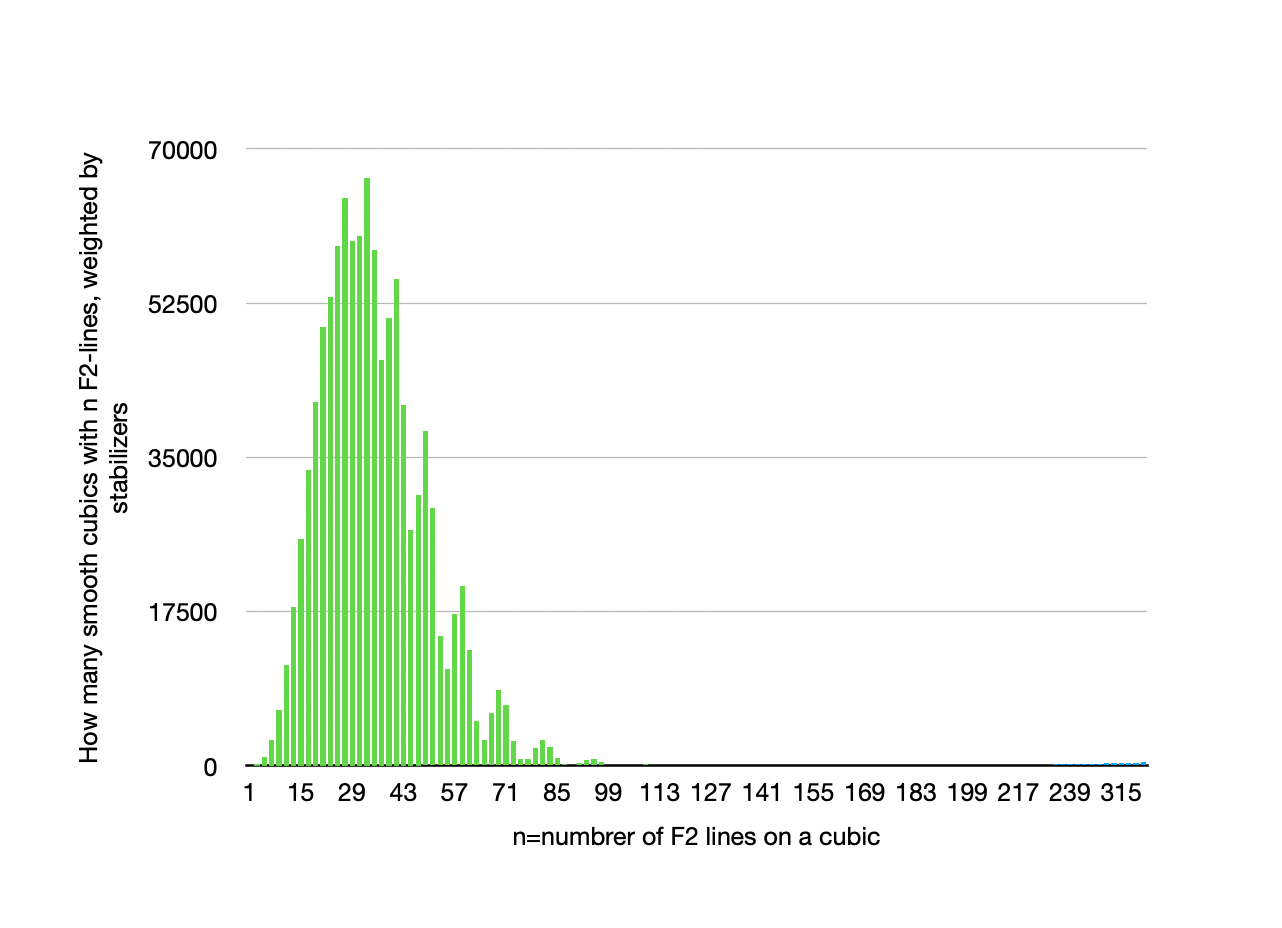}  
  \caption{Stacky counts of the number of cubics containing $n$ lines.
  A plot restricting to smooth cubics is given on the right.}
  \label{fig:cubic-lines}
\end{figure}

\begin{remark}
The maximal number of $\F_2$ lines on a smooth cubic fourfold is
$315$.  The extremal cubic fourfold $X_1$ in Remark~\ref{rem:DKHK} is
the unique smooth cubic fourfold over $\F_2$ with $315$ lines.
\end{remark}

\subsection{Planes on cubic fourfolds}

We also compute the complete set of $\F_2$-planes on every cubic
fourfold over $\F_2$. In contrast to the case of lines, not every
cubic has a plane; indeed, cubic fourfolds containing plane live on the
Noether--Lefschetz divisor $\mathcal{C}_8 \subset \mathcal{C}$. There
are $2\,116\,029$ cubic fourfolds, or $56.90\%$ up to isomorphism,
containing at least one $\F_2$-plane, of which $702,153$ are smooth,
or $65.65\%$ of the smooth cubic fourfolds up to
isomorphism. Figure~\ref{fig:cubic-planes} shows the histograms
recording how many cubic fourfolds (respectively, smooth cubic fourfolds)
contain a fixed number of $\F_2$-planes.

\begin{figure}[h]
  \centering
  \includegraphics[scale=0.3]{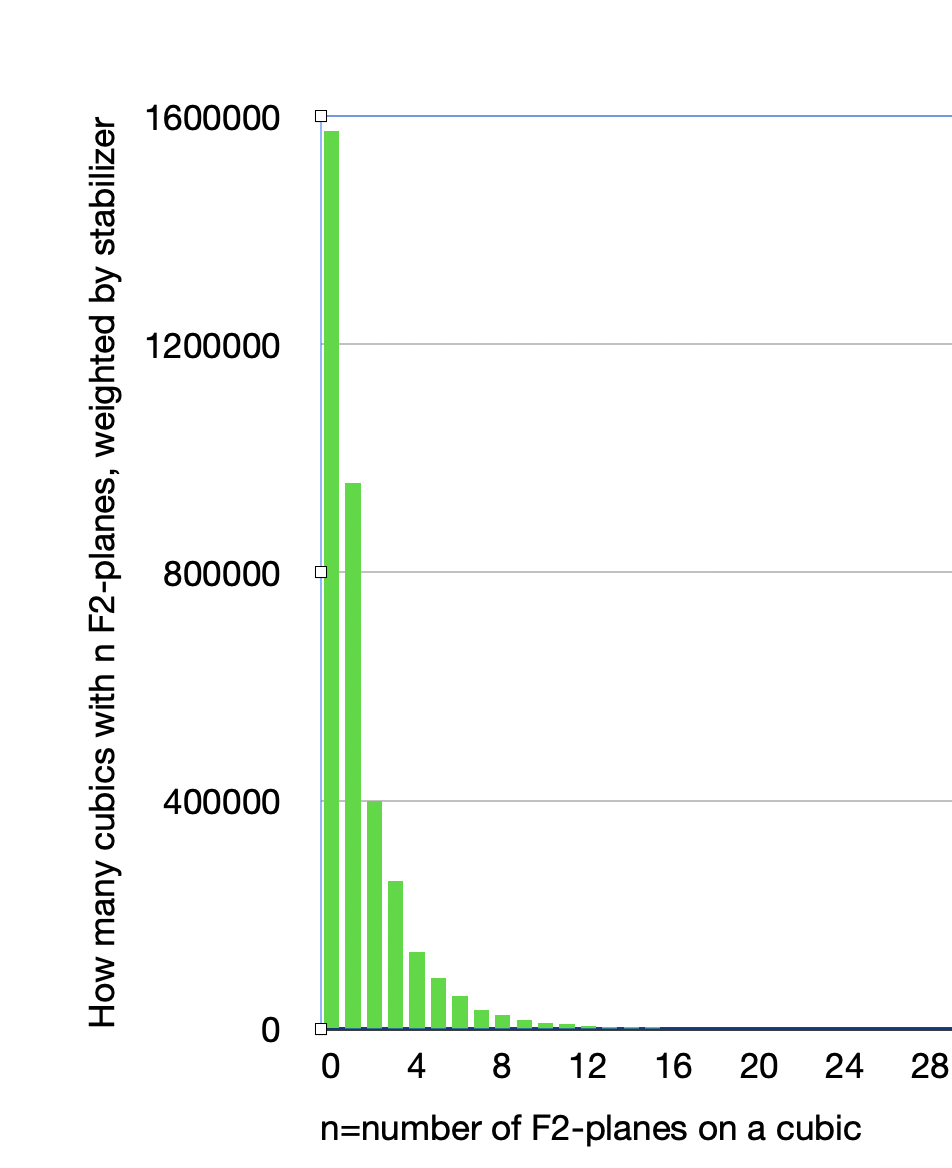} \text{ {} \ \ \ \ {}}
  \includegraphics[scale=0.48, trim=0 1.55cm 0 0]{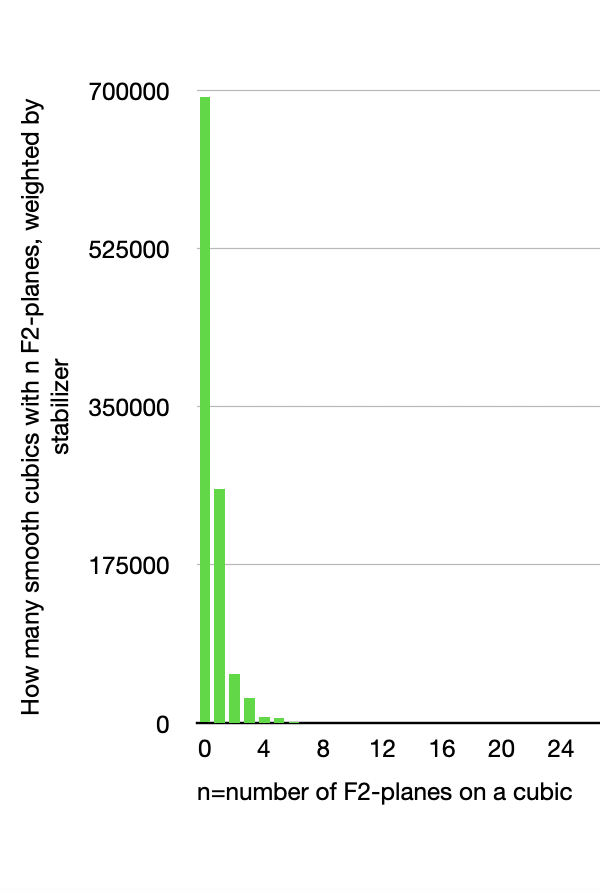}
  
  \caption{Stacky counts of the number of cubics containing $n$ planes.
  A plot restricting to smooth cubics is given on the right.}
  \label{fig:cubic-planes}
\end{figure}

As mentioned in the introduction, the rationality problem for cubic fourfolds has been a primary motivation for their study over the last 50 years. If $X$ contains a pair of disjoint $\F_2$-planes $P_1$ and $P_2$, then $X$ is birational to $P_1 \times P_2$. Counting the cubics in our database with pairs of disjoint $\F_2$ planes gives a lower bound on the number of rational cubic fourfolds over $\F_2$:
\begin{comp}
There are $429\, 744$ isomorphism classes of cubic fourfolds in
$\mathcal{C}(\F_2)$ containing two disjoint $\F_2$-planes, and $36\, 572$ of these cubics are smooth.
\end{comp}

%%% SECTION
\section{Zeta functions of cubic fourfolds}
\label{sec:zeta}

In this section we explain how we compute the zeta functions of all smooth cubic fourfolds over $\F_2$. The reader interested in the results of the computation is advised to skip to~\S\ref{ss:mainresultsonzeta}.  

\subsection{Computational methods}
\label{ss:zetamethods}

Let $q=2$, and let $F\colon X_{\overline{\F}_q} \to
X_{\overline{\F}_q} $ be the relative Frobenius endomorphism.  By the
Weil conjectures, the zeta function for a smooth cubic fourfold
$X/\F_q$ is given by
\[  Z_X(q^{-s} ) = \frac{1}{(1-q^{-s})(1-q^{1-s})(1-q^{2-s})Q_X(q^{-s})(1-q^{3-s})(1-q^{4-s})} = \frac{Z_{\P^4}(q^{-s})}{Q_X(q^{-s})}
\]
where the
interesting factor $Q_X(t)$ is given by, for an odd prime $\ell$,
\[
Q_X(t) = \det\bigl(\text{Id}-tF^*\mid H^4_{\text{\'{e}t,prim}}(X_{\overline{\F}_2}, \Q_\ell) \bigr).
\]

Let us now describe how we efficiently compute the zeta functions of cubic fourfolds. Our code computes the Weil polynomials 
\[
P_X(t)=\det\bigl( F^* - t \text{Id} \mid  H^4_{\text{\'{e}t,prim}}(X_{\overline{\F}_2},
\Q_\ell(2)) \bigr)
\]
for each isomorphism class of smooth cubic fourfold $X/\F_2$ in our
database; each $P_X(t)$ is a monic, degree 22 polynomial with
coefficients in $\frac{1}{2}\Z$ whose roots lie on the unit circle.
The Weil polynomial $P_X(t)$ is related to $Q_X(t)$ via $P_X(t) = \pm
Q_X(t/4)$ where the sign is the sign occurring in the functional
equation, see~\eqref{eqn:functional} below. Using a slight adaptation
of the algorithm of Addington--Auel in~\cite{addington_auel} (described
below), we can efficiently compute the point counts of smooth (and
even mildly singular) cubic fourfolds over $\F_2$.  Then for each
smooth cubic $X$, we compute the first 11 nonleading coefficients of
$P(t)$ using the points counts $\abs{X(\F_{2^k})}$ for $1 \le k \le 11$.

The 11 remaining coefficients are computed by leveraging the functional equation 
\begin{equation}\label{eqn:functional}
P_X(t) =(-1)^{\epsilon} t^{11} P_X(t^{-1})
\end{equation}
to fully determine $P_X(t)$.  However, to determine the sign $\epsilon \in \{0,1\}$ 
of the functional equation, we use work of T.\ Saito~\cite{saito} to
relate it to the {\it divided discriminant} $\text{disc}_d(F)$ of an
integral homogeneous cubic lift $f\in \Z[x_0, \ldots, x_5]$ 
(see~\cite[Definition 2.2]{saito} for the definition of
$\text{disc}_d(F)$).

\begin{theorem}[{\cite[\S4]{saito}}]
Let $X$ be a smooth cubic fourfold over $\F_2$, and let $f \in \Z
[x_0, \ldots, x_5]$ be a lift of a defining equation for $X$. Then the
sign of the functional equation in the zeta function of $X$ is
$(-1)^{(\mathrm{disc}_d(F)+1)/{4}}$.
\end{theorem}

Applying Saito's criterion to the smooth cubics in our database, one finds that nearly half of all smooth cubic fourfolds take the positive sign in their functional equation. 

\begin{comp}
Among the $1\, 069\, 562$ isomorphism classes of smooth cubic
fourfolds over $\F_2$, there are exactly $531\, 334$, or about $49.8\%$,  for which the sign of the functional equation is $+1$.
\end{comp}

\subsection{Addington--Auel point counting algorithm}
For the sake of documentation, we describe our adaptations to the
point counting algorithm of Addington and
Auel~\cite[\S3]{addington_auel}, the idea of which itself goes
back (for counting points on cubic threefolds) to
Bombieri--Swinnerton-Dyer~\cite{bombieri} and was used by
Debarre--Laface--Roulleau~\cite[\S4.3]{debarre_laface_roulleau}.
Our improvements are as follows. First, we no longer require that the
cubic fourfold $X$ contains an $\F_2$-line not also contained in a
plane $P \subset X \times \overline{\F}_2$.  In fact, if $X$ is
smooth, such lines exists aside from $55$ cubics in our database.
Second, we remove the hypothesis that $X$ is smooth.

The key step in the algorithm is to transform $X$ into a conic fibration.
We denote $\P^5 := \operatorname{Proj} \F_2[y_0, \dots, y_5]$ and
$X \subset \P^5$ an arbitrary cubic fourfold.
We let $\ell \subset X$ be a rational line (such a line must exist by Lemma~\ref{lem:containsline}) and change coordinates so that
$\ell = Z(y_0,y_1,y_2,y_3)$. Then the defining equation of $X$ is
\begin{equation}
  \label{eq:conic-fibration}
  \begin{alignedat}{9}
    & A(y_0, \ldots, y_3)y_4^2 && + B(y_0, \ldots, y_3)y_4y_5 && + C(y_0, \ldots, y_3)y_5^2 \\
    + & D(y_0, \ldots, y_3)y_4 && + E(y_0, \ldots, y_3)y_5 && + F(y_0, \ldots, y_3) &= 0
  \end{alignedat}
\end{equation}
where $A, \dots, F$ are homogeneous polynomials of degrees
$1,1,1,2,2,3$, respectively.  If $P \supset \ell$ is a plane, then
equation~\eqref{eq:conic-fibration} shows $P$ meets $X$ along $\ell
\cup C$, where $C$ is a plane conic.

We consider the projection $\phi\: X \dashrightarrow \P^3$ away from
$\ell$, as well as the resolution $\phi\: \wtilde X \rightarrow \P^3$
of $\phi$ and the blowing-down morphism $\pi\: \wtilde X \to
X$. Because $\phi$ is a linear projection, we have that $\pi^{-1}(x)$
is a linear space for every point $x \in X$. Specifically,
\[
  \pi^{-1}(x) \isom
  \begin{cases}
    \P^0 & \text{if } x \not \in \ell \\
    \P^2 & \text{if } x \in \ell, x \not \in X^\sing \\
    \P^3 & \text{otherwise}
  \end{cases}
  .
\]
Notice that $x \in \ell$ is a singularity if and only if the point
$(x_4 : x_5 : 0) \in \P^2$ is a basepoint of the $3$-dimensional
family of conics defined in \eqref{eq:conic-fibration}.  Thus,
depending on $A,B,C$, we have one of the following cases
\begin{itemize}
\item $0$ basepoints (on the hyperplane at infinity),
\item $1$ basepoint 
\item $2$ basepoints (both defined over $\F_2$)
\item $2$ basepoints (neither defined over $\F_2$)
\item A line of basepoints. ($A = B = C = 0$.)
\end{itemize}

\begin{algorithm}\label{algorithm:aa}
  \texttt{CountPoints}($X$, $q$) \quad 
  (Adapted from \cite[\S 3]{addington_auel})

\noindent\textbf{Input:}
  \begin{itemize}
  \item
    A cubic fourfold $X$

  \item
    $q = 2^r$
  \end{itemize}
\textbf{Steps:}

  \begin{enumerate} 
  \item
    Choose an $\F_2$-line $\ell \subset X$ (guaranteed by Lemma~\ref{lem:containsline}).
    
  \item
    The projection away from $\ell$ yields a morphism
    $\phi \colon \text{Bl}_\ell(X) \to \P^3$ whose generic fiber is a conic, as in
    equation~\eqref{eq:conic-fibration}. 

  \item
    If $A = B = C = D = E = 0$, then $\phi$ gives $X$ the structure of a cone over a cubic
    surface $Y$. In this case \\

    \noindent
    \textbf{return} $\abs{Y(\F_q)} \cdot q^2 + q + 1$. \\
    
  \item
    Compute $\abs{\text{Bl}_\ell(X)(\F_q)}$ by counting
    $\F_q$-points in each fiber of $\phi$: let
    \[
      \Delta = Z(AE^2 + B^2F + CD^2 - BDE) \subseteq \P^3
    \]
    be the discriminant subscheme
    parametrizing the fibers of $\phi$ which are either planes or
    singular conics.  For each $y \in \Delta(\F_q)$, the point-count
    $\abs{(\phi^{-1}(y))(\F_q)}$ is determined as follows:\\
    
    \begin{itemize}
    \item
      if $y \in \Delta$ and $y \notin Z(A, \ldots, F)$, then
      $\phi^{-1}(y)$ is a singular plane conic over $\F_q$.
      It has $0$, $q+1$, or $2q+1$ $\F_q$-points, which can be determined by
      its rank and Arf invariant. \\
      
    \item
      if $y \in Z(A, \ldots, F)$, then $\phi^{-1}(y)$ is a $2$-plane over
      $\F_q$ with $1+q+q^2$ $\F_q$-points. \\
    \end{itemize}
    Then
    \[
      \abs{\text{Bl}_\ell(X)(\F_q)} = \sum_{x \in \Delta(\F_q)}\abs{(\phi^{-1}(x))(\F_q)}
      + (q+1)(1+q+q^2+q^3-\abs{\Delta(\F_q)} ).
    \]

  \item
    Determine the exceptional divisor $E$ of $\pi$.
    \begin{itemize}
    \item
      If $x \in \ell$ is not a singular point of $X$, $\abs{\pi^{-1}(x)(\F_q)} = 1 + q + q^2$.
    \item
      If $x \in \ell$ is a singular point, $\abs{\pi^{-1}(x)(\F_q)} = 1 + q + q^2 + q^3$. \\
    \end{itemize}
    In the case that $\ell \subset X^{\sm}$, we have $\abs{E(\F_q)} = (1+q)(1+q+q^2)$. Thus,
    we may compute
    \[
      \abs{\text{Bl}_\ell(X)(\F_q)} = \abs{X(\F_q)} + \abs{E(\F_q)} - \abs{\P^1(\F_q)}.
    \]
    \textbf{return} $\abs{X(\F_q)}$.
  \end{enumerate}

\end{algorithm}

As in the original algorithm in~\cite{addington_auel}, we gain
significant speedup in the enumeration of $\Delta(\F_q)$ by taking the
projection $\Delta \backslash \{p\} \to \P^2$ away from a singular
point $p \in \Delta^{\sing}(\F_2)$ and enumerating over $\P^2(\F_q)$
points. We found in the process of running the point-counting
algorithm that every smooth cubic fourfold over $\F_2$ contains an
$\F_2$-line such that $\Delta^{\sing}(\F_2)$ is nonempty, answering a
question of~\cite{addington_auel} (see footnote~4 of \emph{loc. cit.}).

In fact, there is a simple way to confirm this is true. Choose a line
$\ell \subset X$ and write $X$ as in
Equation~\eqref{eq:conic-fibration}. Since $X$ is smooth, we have that
$A,B,C$ are independent linear forms, so there is a unique
$p = (p_0: p_1 : p_2 : p_3)$ at the intersection of the zero loci in
$\P^3$. The fiber $\phi^{-1}(p)$ in $X$ is given by
\[
  \ell'\: D(p) y_4 + E(p) y_5 + F(p) = 0, \ \ (y_0 : y_1 : y_2 : y_3) = p.
\]
That is, $\ell'$ is a line contained within $X$. The conic fibration
associated to $\ell'$ contains a rational conic of rank $1$ (the
associated fiber is the original line $\ell$ with multiplicity~$2$).

\subsection{Census of the zeta functions}
\label{ss:mainresultsonzeta}

Our computations of all the zeta functions of smooth cubic fourfolds
over $\F_2$ yield the following result.

\begin{comp}\label{comp:zetacount}
There are $86\,472$ distinct zeta functions realized among the
$1\,069\,562$ isomorphism classes of smooth cubic fourfolds over
$\F_2$.
\end{comp}

In doing the computation above, we also verified a conjecture of
Elsenhans and Jahnel~\cite[Theorem 1.9]{elsenhans_jahnel:duke} in the
case of cubic fourfolds over $\F_2$.
  
 \begin{theorem}Let $X$ be a smooth cubic fourfold defined over $\F_2$, and $P_X(t)$ its primitive Weil polynomial. Then $2P_X(-1)$ is an integer square.
\end{theorem}

  \subsection{The $K3$-part of the Weil polynomial of a cubic fourfold}

If $(1-t)$ divides the Weil polynomial $P_X(t)$ of some cubic fourfold $X$, then $P_X(t)/(1-t)$ is a degree $21$ Weil polynomial, which we call the {\it $K3$-part} of the Weil polynomial of the cubic fourfold.

\begin{comp}
\label{comp:compare}
We compared the Weil polynomials $P_X(t)/(1-t)$ against the list
generated by Kedlaya--Sutherland~\cite[Computation~3(c)]{kedlaya_sutherland-census} of Weil polynomials of degree 21
which are of ``K3-type'', finding that there are $71\,476$
$K3$-type Weil polynomials which are the $K3$-part of the Weil
polynomial of some cubic fourfold over $\F_2$. 
\end{comp}

Computation~\ref{comp:compare} shows that cubic fourfolds realize
$K3$-type Weil polynomials that are not realized by any quartic $K3$
over $\F_2$; indeed, there are only $52\,755$ degree 21 Weil
polynomials arising from these quartic
surfaces~(\cite[Computation~4(c)]{kedlaya_sutherland-census}).

These $K3$-type Weil polynomials arising form cubic fourfolds are
sometimes (but not always) explained by the phenomenon of {\it
associated $K3$ surfaces} mentioned in the introduction. For instance,
whenever there is a geometric construction of an associated (twisted)
$K3$ surface defined over the ground field $k$, the $K3$-part of the
Weil polynomial of the cubic fourfold arises from an honest $K3$
surface over $k$. This is part of a more general phenomenon: whenever
there is a $k$-linear Fourier--Mukai equivalence between the $K3$
category of $X$ (as defined in~\cite{kuznetsov}, see
\cite{huybrechts}) and the derived category of (twisted) coherent
sheaves on a $K3$ surface, Fu and Vial~\cite{fu_vial} show that the
transcendental zeta functions of $S$ and $X$ agree (in fact, they have
isomorphic rational Chow motives). We refer the interested reader to
the excellent survey of Hassett~\cite{hassett-survey} for more on
associated $K3$ surfaces of cubic fourfolds.

\subsection{Newton polygons}

Having tabulated the zeta functions of the smooth cubic fourfolds over
$\F_2$, we can determine their Newton polygons. The Newton polygon of
a cubic fourfold over a finite field is determined by its height $h$,
which can be any integer $1 \le h \le 10$, or $h=\infty$. We find
cubic fourfolds over $\F_2$ of every possible height.

\begin{theorem}
Each Newton stratum in the moduli space of smooth cubic fourfolds contains $\F_2$-points.
\end{theorem}

A cubic fourfold is called {\it ordinary} if $h=1$ and {\it supersingular} if $h=\infty$. 

\begin{comp}
There are $8688$ supersingular cubic fourfolds and $533, 262$ ordinary cubic fourfolds up to isomorphism over $\F_2$.
\end{comp}

\begin{figure}[h]
{\small
  \[
    \begin{array}{c|ccccccccccc}
      h  &1 & 2 & 3 & 4& 5 & 6 & 7 & 8 & 9 & 10 & \infty \\[5pt] \hline  \\ [-1pt]
      \# &533262 & 267355 & 131922 & 66974 & 31806 & 16041 & 6901 & 4575 & 1301 & 737 & 8688 

    \end{array}
  \]
}
\caption{Heights of isomorphism classes of cubic fourfolds over $\F_2$  }
\label{fig:heights}
\end{figure}

\subsection{Codimension $2$ algebraic cycles on cubic fourfolds}

The Tate conjecture for $K3$ surfaces over finite fields of
characteristic~$2$ has been proved by
Ito--Ito--Koshikawa~\cite{ito_ito_koshikawa} and Kim--Madapusi
Pera~\cite{kim_madapusipera},~\cite{madapusipera-erratum}. We now
explain how methods used to prove these results, as well as results
for Gushel--Mukai varieties by Fu and Moonen~\cite{fu_moonen}, lead to
a proof of the Tate conjecture for codimension $2$ cycles on cubic
fourfolds over $\F_2$. Since this result is surely known to the
experts, we only provide a sketch of the proof to fill an existing gap
in the literature.

\begin{theorem}\label{thm:tateconj}
Let $X$ be a smooth cubic fourfold over a finite field $k$ of
characteristic~$2$. Then the cycle class map induces an isomorphism
\[
\CH^2(X) \otimes \Q_\ell \to H^4_{\mathrm{\acute{e}t}} (\overline{X},
\Q_\ell(2))^{\mathrm{Gal}(\overline{k}/k)}.
\]
\end{theorem}

\begin{proof}
We use Madapusi Pera's approach to proving the Tate conjecture for
codimension 2 cycles on smooth cubic fourfolds outlined in
\cite[\S5.14]{madapusipera}, together with the integral $2$-adic
models of Shimura varieties constructed in \cite{kim_madapusipera},
and the revised approach in \cite{madapusipera-erratum}, as adopted in
\cite{fu_moonen}.

Let $\mathscr{C}^{\mathrm{sm}}$ be the stack of smooth cubic fourfolds
over $\Z_{(2)}$.  Let $L \subset L'$ be the abstract lattice of the
primitive part inside the $H^4$ of a cubic fourfold.  Then $L$ is even
with signature $(20,2)$ and discriminant $3$ and $L'$ is odd
unimodular of signature $(21,2)$.  As in~\cite[\S5.14]{madapusipera},
let $\widetilde{\mathscr{C}}^{\mathrm{sm}} \to
{\mathscr{C}}^{\mathrm{sm}}$ be the double cover parameterizing cubic
fourfolds together with a choice of isomorphism $\det(L) \otimes
\Z_\ell \to \det(\langle h^2 \rangle^\perp)$, where $h^2 \in
H^4_{\text{\'{e}t}}(X_{k^s},\Z_\ell(2))$ is the cycle class of the
square of the hyperplane section.  Then by the arguments of
\cite[Proposition 4.15]{kim_madapusipera}, the classical Kuga--Satake
map extends to a morphism $\widetilde{\mathscr{C}}^{\mathrm{sm}} \to
\mathcal{S}(L)$, where $\mathcal{S}(L)$ is a $\Z_{\mathfrak{p}}$-model
of the orthogonal Shimura variety $\mathrm{Sh}(L)$ attached to $L$,
see~\cite[Theorem~3.10]{kim_madapusipera}. Here, one needs to take a
prime $\mathfrak{p}$ lying above $2$ in an extension $E/\Q$ (of degree
at most $2$) that trivializes the quadratic character induced by the
determinant-preserving Galois action on the primitive cohomology
$\langle h^2 \rangle^\perp$,
see~\cite[Remark~6.25~and~\S7.1]{fu_moonen}.  This subtlety, which
arises since the primitive cohomology has even rank, hence its special
orthogonal group contains $\pm \text{id}$, is not directly addressed
in~\cite[\S5.14]{madapusipera}.  However, to prove the Tate conjecture
for codimension 2 cycles on $X$, we are free to take a finite
extension of $k$, cf.\ \cite[\S~2,~p.~580]{totaro}, namely, the
residue field of $E_\mathfrak{p}$.

Following the strategy in \cite{madapusipera}, the main step is to
show that the map $\widetilde{\mathscr{C}}^{\mathrm{sm}} \to
\mathcal{S}(L)$ is \'etale.  We
remark that the de Rham realization of the universal lattice
$\bm{L}_{\mathrm{dR}} \subset \bm{L}'_{\mathrm{dR}}$ is a vector
subbundle since any polarization (having self-intersection 3) is
primitive.  Similarly, we appeal to
\cite[Lemma~1.10]{madapusipera-erratum} to show that the map induced on
de Rham realizations extends to an isometry to filtered vector bundles
\[
\alpha_{\mathrm{dR}} : \bm{L}_{\mathrm{dR}}(-2) \to \bm{H}_{\mathrm{prim},\mathrm{dR}}^4
\]
over $\widetilde{\mathscr{C}}^{\mathrm{sm}}$.  
The rest of the proof proceeds as in \cite[\S5.14]{madapusipera},
since $\widetilde{\mathscr{C}}^{\mathrm{sm}}$ is a smooth Artin stack,
as proved by Levin~\cite[\S3]{levin}
(cf.\ \cite[Proof~of~Theorem~5.15]{madapusipera}), 
and the deformation theory of $\widetilde{\mathscr{C}}^{\mathrm{sm}}$
is controlled by the degree 4 part of the Hodge filtration on
$\bm{H}_{\mathrm{prim},\mathrm{dR}}^4$.
\end{proof}

For a Weil polynomial $P(t)$ of a cubic fourfold $X$, we write $P(t)
=P_{\mathrm{cyc}}(t) P_{\mathrm{non-cyc}}(t)$ where
$P_{\mathrm{cyc}}(t)$ is the product of all cyclotomic factors of
$P(t)$. If $(t-1)^m$ exactly divides $P_{\mathrm{cyc}}(t)$ and
$\deg P_{\mathrm{cyc}}(t) = n$, then we have, as a direct consequence
of the Tate conjecture, that
\[
\rk \CH^2(X) = m \quad \text{and} \quad
\rk \CH^2(\overline{X}) = n
\]
Thus, we can report the following outcome from our computation of all zeta functions.

\begin{theorem}
If $X$ is an smooth cubic fourfold over $\F_2$, then the algebraic
rank $r=\rk \CH^2(X)$ can be any integer $1 \le r \le 10$, or
$r=12,16$ and, furthermore, there are ordinary cubic fourfolds of
every algebraic rank up to rank $10$.
\end{theorem}

\begin{remark}
In fact, the extremal cubic fourfold $X_1$ in Remark~\ref{rem:DKHK} is
the unique smooth cubic fourfold over $\F_2$ with algebraic rank $16$.
\end{remark}

\begin{theorem}
If $X$ is an smooth cubic fourfold over $\F_2$ then the algebraic rank $\rk
\CH^2(\overline{X})$ can be any odd number $\le 23$, and all such
ranks $\le 21$ are realized by ordinary smooth cubic fourfolds.
\end{theorem}

We remark that since the Tate conjecture holds for a supersingular
cubic fourfold $X$, the algebraic cycles $\CH^2(\overline{X})$
span $H^4_{\mathrm{\acute{e}t}}(\overline{X}, \Q_\ell(2))$ and so any supersingular
cubic fourfold has geometric rank $\rk \CH^2(\overline{X}) = 23$.

The tables below summarize our computations of the ranks of the algebraic and geometric Chow groups of smooth cubic fourfolds over $\F_2$.
\begin{figure}[h]
{\small
  \[
\renewcommand{\arraystretch}{1.1}
    \hspace*{-1.45cm}\begin{array}{c|cccccc}
      \rk \CH^2(\overline{X}) &1 & 3 & 5 & 7 & 9 & 11 \\
      \text{how many}&107552 & 254144 & 153410 & 179596 & 107911 & 98978\\[4pt] \hline \\[-9pt]
      \rk \CH^2(\overline{X}) & 13 & 15 & 17 & 19 & 21 & 23 \\
      \text{how many} & 61054 & 50777 & 27339 & 14588 & 5525 & 8688\\[-13pt] 
    \end{array}
  \]
}
\caption{Rank of the group of geometric cycles $\CH^2(\overline{X})$}
\label{fig:chow2geom}
\end{figure}
\begin{figure}[h]
{\small
  \[
  \renewcommand{\arraystretch}{1.1}
   \begin{array}{c|cccccccc} \setstretch{0.1}
      \rk \CH^2(X) &1 & 2 & 3 & 4& 5 & 6 & 7 & 8\\
      \text{how many}&232218 & 426619 & 273007 & 106035 & 25521 & 5377 & 581 & 178 \\[4pt] \hline \\[-9pt]
      \rk \CH^2(X) & 9 & 10 & 11 & 12 & 13 & 14 & 15 & 16 \\
      \text{how many} & 7 & 13 & 0 & 5 & 0 &0 & 0 & 1\\[-13pt]
    \end{array}
  \]
}
\caption{Rank of the group of algebraic cycles $\CH^2(X)$  }
\label{fig:chow2}
\end{figure}

\vspace*{1in}

\end{document}